\documentclass[12pt]{amsart}
\usepackage{a4}
\usepackage{amssymb}
\usepackage{color}

\newtheorem{theorem}{Theorem}[section]
\newtheorem{Prop}[theorem]{Proposition}
\newtheorem{Thm}[theorem]{Theorem}
\newtheorem{Lem}[theorem]{Lemma}
\newtheorem{Cor}[theorem]{Corollary}

\theoremstyle{definition}

\newtheorem{Def}[theorem]{Definition}

\theoremstyle{remark}
\newtheorem{remark}[theorem]{Remark}

\numberwithin{equation}{section}

\newcommand{\R}{{\mathbb R}}

\newcommand{\Z}{{\mathbb Z}}

\newcommand{\LG}{{\mathfrak g}}
\newcommand{\LH}{{\mathfrak h}}

\newcommand{\UU}{{\mathfrak U}}

\newcommand{\Aut}{\mbox{\rm Aut}}

\newcommand{\inner}[2]{\langle #1 , #2 \rangle} 

\newcommand{\GL}{\mathrm{GL}} 
\newcommand{\OO}{\mathrm{O}} 
\newcommand{\RH}{\R \mathrm{H}} 
\newcommand{\MM}{\mathfrak{M}} 
\newcommand{\PM}{\mathfrak{PM}}

\newcommand{\transpose}[1]{{}^t \hspace{-1pt}{#1}} 

\begin{document}

\title[The moduli spaces of left-invariant pseudo-Riemannian metrics]
{On the moduli spaces of left-invariant pseudo-Riemannian metrics on Lie groups}

\author{Akira Kubo} 
\address[A.\ Kubo]{Faculty of Economic Sciences, Hiroshima Shudo University, 
Hiroshima 731-3195, Japan} 
\email{akubo@shudo-u.ac.jp} 

\author{Kensuke Onda} 
\address[K.\ Onda]{Faculty of Teacher Education, Shumei University, Yachiyo 276-0003, Japan} 
\email{onda@mailg.shumei-u.ac.jp} 

\author{Yuichiro Taketomi}
\address[Y.\ Taketomi]{Department of Mathematics, Hiroshima University, 
Higashi-Hiroshima 739-8526, Japan}
\email{y-taketomi@hiroshima-u.ac.jp}

\author{Hiroshi Tamaru}
\address[H.\ Tamaru]{Department of Mathematics, Hiroshima University, 
Higashi-Hiroshima 739-8526, Japan}
\email{tamaru@math.sci.hiroshima-u.ac.jp}

\thanks{
The first author was supported in part by Grant-in-Aid for JSPS Fellows (11J05284). 
The fourth author was partially supported by KAKENHI (24654012, 26287012).
} 

\begin{abstract} 
In this paper, 
we formulate a procedure to obtain a generalization of Milnor frames for 
left-invariant pseudo-Riemannian metrics on a given Lie group. 
This procedure is an analogue of the recent studies on left-invariant Riemannian metrics, 
and is based on the moduli space of left-invariant pseudo-Riemannian metrics. 
As one of applications, 
we show that any left-invariant pseudo-Riemannian metrics of arbitrary signature 
on the Lie groups of real hyperbolic spaces have constant sectional curvatures. 
\end{abstract} 

\maketitle


\section{Introduction}

Left-invariant Riemannian and pseudo-Riemannian metrics on Lie groups 
have been studied actively. 
In particular, they provide a lot of interesting examples of distinguished metrics, 
for examples, Einstein and Ricci soliton metrics. 
It is a natural and important problem to determine whether a given Lie group $G$ admits 
some distinguished left-invariant (pseudo-)Riemannian metrics or not. 
Both of the Riemannian and pseudo-Riemannian cases are interesting, 
but sometimes the properties and methodologies are different. 
Recently many results on the Riemannian case have been obtained 
(just as examples, see \cite{HT, HTT, KTT, TT} and references therein). 
However, the studies on the pseudo-Riemannian case seem to be still developing. 

In the studies on left-invariant metrics on Lie groups, 
particularly on the existence and nonexistence of some distinguished metrics, 
one of the difficulties may come from the fact that 
the space of left-invariant metrics has a large dimension. 
Since left-invariant metrics on $G$ and inner products on its Lie algebra $\LG$ are corresponding, 
the following kind of theorem would be helpful: 
\begin{itemize} 
\item[] 
For every inner product $\inner{}{}$ of signature $(p,q)$ on $\LG$, 
there exists a basis of $\LG$, 
which is pseudo-orthonormal with respect to $\inner{}{}$ up to scalar, 
and the bracket relations among the elements of the basis can be written with 
relatively smaller number of parameters. 
\end{itemize} 
This kind of theorem is called a \textit{Milnor-type theorem} in \cite{HTT}. 
The name comes from the famous result by Milnor (\cite{Mi}), 
who obtained this kind of theorems for 
left-invariant Riemannian metrics on all three-dimensional unimodular Lie groups 
(the obtained bases are called the \textit{Milnor frames}). 
Milnor-type theorems have also been known 
for left-invariant Riemannian metrics on Lie groups with $\dim \leq 4$ 
(\cite{Chebarikov, HL, KN1, KN2}), 
and for left-invariant pseudo-Riemannian (Lorentzian) metrics on Lie groups with $\dim = 3$ 
(\cite{CP, R92, RR06}). 
Recently, for left-invariant Riemannian metrics, 
a general procedure to obtain Milnor-type theorems has been formulated in \cite{HTT}. 
It is based on the moduli space of left-invariant Riemannian metrics on $G$, 
where the moduli space is 
defined as the orbit space of the action of $\R^\times \Aut(\LG)$ 
on the space $\MM(G)$ of left-invariant Riemannian metrics on $G$. 
In fact, an expression of the moduli space derives a Milnor-type theorem for $G$. 

In this paper, we formulate a general procedure to obtain Milnor-type theorems 
for left-invariant pseudo-Riemannian metrics on $G$. 
It is also based on the moduli space, that is, 
the orbit space of the action of $\R^\times \Aut(\LG)$ on the space $\MM_{(p,q)}(G)$ 
of left-invariant pseudo-Riemannian metrics of signature $(p,q)$ on $G$. 
In fact, the procedure itself is a straightforward analogue to the Riemannian case. 
However, by applying this procedure to two particular Lie groups, 
one can observe a different phenomena between Riemannian and pseudo-Riemannian cases. 
The reason may be because 
\begin{align} 
\MM_{(p,q)}(G) \cong \GL_{p+q}(\R) / \OO(p,q) 
\end{align} 
is a pseudo-Riemannian symmetric space if $p,q \geq1$, 
although $\MM(G) = \MM_{(n,0)}(G)$ is a Riemannian symmetric space, 
where $n = p+q = \dim G$. 

The first example of Lie group to which we apply our procedure is $G_{\RH^n}$, 
the Lie group of real hyperbolic space $\RH^n$. 
Recall that $G_{\RH^n}$ is the connected and simply-connected Lie group with Lie algebra $\LG_{\RH^n}$, 
where 
\begin{align} 
\LG_{\RH^n} = \mathrm{span} \{ e_1, \ldots , e_n \} \mbox{ with } 
[e_1, e_j] = e_j \ (j \in \{ 2, \ldots, n \}) . 
\end{align} 
Note that $G_{\RH^n}$ coincides with the solvable part of the Iwasawa decomposition of $\mathrm{SO}^0(n,1)$, 
and hence acts simply-transitively on $\RH^n$. 
By applying our procedure to $\LG_{\RH^n}$, 
we have the following Milnor-type theorem for $G_{\RH^n}$.

\begin{Thm} 
\label{thm:MTT1} 
Let $p,q \geq 1$, and 
$\inner{}{}$ be an arbitrary inner product of signature $(p,q)$ on $\LG_{\RH^{p+q}}$. 
Then, there exist $k>0$, 
$\lambda \in \{ 0, 1, 2 \}$, and 
a pseudo-orthonormal basis $\{ x_1 , \ldots , x_{p+q} \}$ with respect to $k \inner{}{}$ 
whose bracket relations are given by 
\begin{enumerate} 
\item[(i)] 
$[x_1 , x_i] = x_i$ \ $(i \in \{ 2 , \ldots , p+q-1 \})$, 
\item[(ii)] 
$[x_1 , x_{p+q}] = - \lambda x_1 + x_{p+q}$, 
\item[(iii)] 
$[x_i , x_{p+q}] = - \lambda x_i$ \ $(i \in \{ 2 , \ldots , {p+q}-1 \})$. 
\end{enumerate} 
\end{Thm} 

Note that the bracket relations among the elements of the basis $\{ x_1 , \ldots , x_{p+q} \}$ 
contain only one discrete parameter $\lambda \in \{ 0, 1, 2 \}$. 
This enables us to calculate the curvature 
of an arbitrary pseudo-Riemannian metrics on $G_{\RH^{p+q}}$. 

\begin{Cor} 
\label{cor:cor1} 
Let $p,q \geq 1$. 
Then, 
every left-invariant pseudo-Riemannian metric of signature $(p,q)$ on $G_{\RH^{p+q}}$ 
has constant sectional curvature. 
Furthermore, any real number can be realized as the constant sectional curvature of 
a left-invariant pseudo-Riemannian metric of signature $(p,q)$ on $G_{\RH^{p+q}}$. 
\end{Cor} 

Note that a left-invariant Riemannian (that is, $q=0$) metric on $G_{\RH^{p+q}}$ 
is unique up to isometry and scaling, and has constant negative curvature 
(\cite{Mi}, see also \cite{KTT, L}). 
We also note that the Lorentzian (that is, $q=1$) case of this result has been known by Nomizu (\cite{Nomizu}). 
Hence, our argument gives an alternative proof of his result in terms of a Milnor-type theorem, 
and generalizes it to the case of arbitrary signature. 

The second example to which we apply our procedure is the three-dimensional Heisenberg group $H_3$. 
We show the following Milnor-type theorem for left-invariant Lorentzian metrics on $H_3$. 
Denote by $\LH_3$ the Lie algebra of $H_3$, 
the three-dimensional Heisenberg Lie algebra. 

\begin{Thm} 
\label{thm:MTT2} 
Let $\inner{}{}$ be an arbitrary inner product of signature $(2,1)$ on $\LH_3$. 
Then, there exist $k>0$, $\lambda \in \{ 0, 1, 2 \}$, and 
a pseudo-orthonormal basis $\{ x_1 , x_2 , x_3 \}$ with respect to $k \inner{}{}$ such that 
\begin{align}
[x_1 , x_2] = \lambda (x_1 + \lambda x_3), \quad 
[x_2 , x_3] = x_1 + \lambda x_3 , \quad 
[x_3 , x_1] = 0 . 
\end{align} 
\end{Thm} 

Note that the bracket relations among the elements of the basis $\{ x_1 , x_2 , x_3 \}$ 
again contain only one discrete parameter $\lambda \in \{ 0, 1, 2 \}$. 
This gives an alternative proof of the classification of 
left-invariant Lorentzian metrics on $H_3$ by Rahmani (\cite{R92}). 
We will also review the curvature properties of these metrics. 
In fact, the metric corresponding to $\lambda = 1$ is flat, 
and other two are Lorentzian Ricci soliton metrics.

This paper is organized as follows. 
In Sections~\ref{section:moduli} and \ref{section:key-thm}, 
we develop a procedure to obtain Milnor-type theorems for a given Lie group $G$. 
We note that this procedure can be applied to any Lie groups, at least theoretically. 
In Section~\ref{sec:action}, 
we study the orbit spaces of the actions of certain maximal parabolic subgroups. 
This may have an independent interest, 
since it could be a prototype for the study of isometric actions on pseudo-Riemannian symmetric spaces. 
In Sections~\ref{section:RHn} and \ref{section:H3}, 
we obtain Milnor-type theorems for the Lie groups $G_{\RH^n}$ and $H_3$, respectively, 
and study the curvature properties.

The authors would like to thank Yoshio Agaoka, Takayuki Okuda, and Takahiro Hashinaga, 
for valuable comments and suggestions.

\section{The moduli space of left-invariant metrics}
\label{section:moduli}

In \cite{KTT}, the notion of the space of left-invariant Riemannian metrics 
on a Lie group up to isometry and scaling has been introduced. 
In this section, we define the analogous notion for 
left-invariant pseudo-Riemannian metrics. 
Throughout this section, 
let $G$ be a Lie group of dimension $n$, 
and $\LG$ be the Lie algebra of $G$. 
We fix a basis $\{ e_1, \ldots, e_n \}$ of $\LG$, 
and identify $\LG \cong \R^{n}$ as vector spaces. 

First of all, we recall the signature of an inner product. 
Let $\inner{}{}$ be an inner product, 
not necessarily positive definite, on $\LG$. 
Then there exists a symmetric matrix $A$ such that 
\begin{align}
\inner{x}{y} = \transpose{x} A y \quad (\forall x, y \in \LG) . 
\end{align} 
Denote by $I_k$ the unit matrix of order $k$, and put 
\begin{align}
I_{p,q} := 
\left( \begin{array}{cc}
I_p & \\ & - I_q 
\end{array} \right) . 
\end{align}
Then, Sylvester's law of inertia yields that 
there exist $g \in \GL_n(\R)$ and a unique pair $(p,q)$ 
with $p,q \in \Z_{\geq 0}$ such that 
\begin{align}
\transpose{g} A g = I_{p,q} . 
\end{align}
This unique pair $(p,q)$ is called the \textit{signature} of $\inner{}{}$. 
Note that $p+q = n$. 

We here define the space of left-invariant metrics. 
We define it in the Lie algebra level, 
since there is a one-to-one correspondence 
between left-invariant (pseudo-Riemannian) metrics on $G$ 
and (indefinite) inner products on $\LG$. 

\begin{Def}
The following set is called the 
\textit{space of left-invariant pseudo-Riemannian metrics} of signature $(p,q)$ on $G$: 
\begin{align}
\MM_{(p,q)} := \{ \inner{}{} : 
\mbox{an inner product of signature $(p,q)$ on $\LG$} \} . 
\end{align} 
\end{Def} 

We next see an expression of $\MM_{(p,q)}$ as a homogeneous space. 
According to the identification $\LG \cong \R^n$, 
we also identify $\GL(\LG) \cong \GL_{n}(\R)$. 
Then $\GL_n(\R)$ acts on $\MM_{(p,q)}$ by 
\begin{align}
g.\inner{\cdot}{\cdot} := \inner{g^{-1} (\cdot)}{g^{-1} (\cdot)} . 
\end{align}
This action is transitive, because of Sylvester's law of inertia. 
Hence, $\MM_{(p,q)}$ can be expressed as a homogeneous space of $\GL_n(\R)$. 
One needs the following inner product, 
which we call the \textit{canonical inner product} of signature $(p,q)$ on $\LG \cong \R^n$: 
\begin{align} 
\label{eq:canonical}
\inner{x}{y}_0 := \transpose{x} I_{p,q} y \quad (x, y \in \LG) . 
\end{align} 

\begin{Prop}
One has a canonical identification 
\begin{align}
\MM_{(p,q)} = \GL_{n}(\R) / \OO(p,q) . 
\end{align} 
\end{Prop}

\begin{proof} 
One can easily see that the isotropy subgroup of $\GL_n(\R)$ 
at $\inner{}{}_0$ coincides with 
\begin{align} 
\OO(p,q) := \{ g \in \GL_n(\R) \mid \transpose{g} I_{p,q} g = I_{p,q} \} . 
\end{align} 
Hence, by virtue of a standard theory of homogeneous spaces, 
one obtains an expression as a homogeneous space. 
\end{proof} 

We now define the moduli space of left-invariant pseudo-Riemannian metrics, 
as the orbit space of a certain group action on $\MM_{(p,q)}$. 
Let us consider the automorphism group of $\LG$, 
\begin{align}
\Aut(\LG) 
:= \{ \varphi \in \GL_n(\R) \mid 
\varphi ([\cdot , \cdot]) = [\varphi(\cdot) , \varphi(\cdot) ] \} . 
\end{align}
Denote by $\R^\times := \R \setminus \{ 0 \}$ the multiplicative group of $\R$. 
The group we consider in this paper is 
\begin{align}
\R^\times \Aut(\LG) 
:= \{ c \varphi \in \GL_n(\R) \mid 
c \in \R^\times , \, \varphi \in \Aut(\LG) \} . 
\end{align}
Since this group is a subgroup of $\GL_n(\R)$, 
it naturally acts on $\MM_{(p,q)}$. 
We denote the orbit through $\inner{}{}$ by $\R^\times \Aut(\LG).\inner{}{}$. 

\begin{Def}
The orbit space of the action of $\R^{\times} \Aut(\LG)$ on $\MM_{(p,q)}$ 
is called the 
\textit{moduli space of left-invariant pseudo-Riemannian metrics} 
of signature $(p,q)$ on $G$, and denoted by 
\begin{align}
\PM_{(p,q)} 
:= \R^{\times} \Aut(\LG) \backslash \MM_{(p,q)} 
:= \{ \R^{\times} \Aut(\LG) . \inner{}{} \mid 
\inner{}{} \in \MM_{(p,q)} \} . 
\end{align}
\end{Def}

Note that the action of $\R^{\times} \Aut(\LG)$ on $\MM_{(p,q)}$ 
gives rise to isometry up to scaling of left-invariant pseudo-Riemannian metrics. 
This follows from a similar argument for the Riemannian case 
(for example, see \cite[Remark~2.3]{KTT}). 

\section{Milnor-type theorems for left-invariant metrics}
\label{section:key-thm}

In this section, we show that an expression of the moduli space $\PM_{(p,q)}$ 
derives a Milnor-type theorem for left-invariant pseudo-Riemannian metrics on $G$. 
The story is analogous to the arguments of \cite{HTT}, 
in which Milnor-type theorems for left-invariant Riemannian metrics have been studied. 

When we express the moduli space $\PM_{(p,q)}$, 
we use the following notion of a set of representatives. 
Recall that $\inner{}{}_0$ denotes the canonical inner product of signature $(p,q)$ on $\LG \cong \R^{p+q}$. 
We define a set of representatives for a group action on 
$\MM_{p+q} = \GL_{p+q}(\R) / \OO(p,q)$. 

\begin{Def} 
Let $H$ be a subgroup of $\GL_{p+q}(\R)$, and consider the action of $H$ on $\MM_{p+q}$. 
Then, a subset $\UU \subset \GL_n(\R)$ is called a 
\textit{set of representatives} 
of this action if the orbit space satisfies 
\begin{align} 
H \backslash \MM_{p+q} = \{ H.(g_0.\inner{}{}_0) \mid g_0 \in \UU \} . 
\end{align} 
\end{Def} 

Note that, by a set of representatives, 
we do not mean that it is a set of complete representatives. 
For example, 
$\UU := \GL_n(\R)$ is a set of representatives for any action. 

In order to formulate a key theorem to obtain Milnor-type theorems, 
we need a pseudo-orthonormal basis. 
Let $\inner{}{}$ be an inner product of signature $(p,q)$ on $\LG$. 
For the later convenience, we put 
\begin{align}
\varepsilon_i := \left\{ 
\begin{array}{cl}
1 & (i \in \{ 1, \ldots , p\} ) , \\ 
-1 & (i \in \{ p+1, \ldots , n \} ) . 
\end{array}
\right. 
\end{align}
Then, a basis $\{ x_1 , \ldots, x_n \}$ of $\LG$ is said to be 
\textit{pseudo-orthonormal} 
with respect to $\inner{}{}$ if it satisfies 
\begin{align}
\inner{x_i}{x_j} = \varepsilon_i \delta_{ij} \qquad 
(\forall i, j \in \{ 1, \ldots, n \}) . 
\end{align}
Here, $\delta_{ij}$ denotes the Kronecker delta. 

\begin{theorem} 
\label{thm:pseudo-key} 
Let $\UU$ be a set of representatives of the action of $\R^\times \Aut(\LG)$ on $\MM_{(p,q)}$. 
Then, for every inner product $\inner{}{}$ of signature $(p,q)$ on $\LG$, 
there exist $k>0$, $\varphi \in \Aut(\LG)$, and $g_0 \in \UU$ such that 
$\{ \varphi g_0 e_1 , \ldots , \varphi g_0 e_n \}$ 
is pseudo-orthonormal with respect to $k \inner{}{}$. 
\end{theorem} 

\begin{proof}
Take an arbitrary inner product $\inner{}{}$ on $\LG$ of signature $(p,q)$. 
Since $\UU$ is a set of representatives, 
there exists $g_0 \in \UU$ such that 
\begin{align}
\inner{}{} \in \R^\times \Aut(\LG).(g_0.\inner{}{}_0) . 
\end{align}
Hence, there exist $c \in \R^\times$ and $\varphi \in \Aut(\LG)$ such that 
\begin{align}
\inner{}{} = (c \varphi).(g_0.\inner{}{}_0) = (c \varphi g_0).\inner{}{}_0 . 
\end{align}
Let us put $k := c^2 > 0$. 
Then we have 
\begin{align}
k \inner{\cdot}{\cdot} 
= k \inner{(c \varphi g_0)^{-1} (\cdot)}{(c \varphi g_0)^{-1} (\cdot)}_0 
= \inner{(\varphi g_0)^{-1} (\cdot)}{(\varphi g_0)^{-1} (\cdot)}_0 . 
\end{align}
Since $\{ e_1, \ldots, e_n \}$ is pseudo-orthonormal with respect to $\inner{}{}_0$, one has 
\begin{align}
k \inner{\varphi g_0 e_i}{\varphi g_0 e_j} 
= \inner{e_i}{e_j} 
= \varepsilon_i \delta_{ij} , 
\end{align}
which completes the proof. 
\end{proof}

When we apply this theorem to a given $\LG$, 
we put $x_i := \varphi g_0 e_i$, 
and study the bracket relations among them. 
Note that $\varphi$ does not give any effects on the bracket relations, 
since it is an automorphism. 
Hence we have only to consider $g_0 \in \UU$. 
In particular, if $\UU$ contains only $l$ parameters, 
then so do the bracket relations among $\{ x_1, \ldots , x_n \}$. 
This is a procedure to obtain Milnor-type theorems. 
We emphasize that this procedure can be applied to any Lie algebras $\LG$. 

\section{Sets of representatives of some actions} 
\label{sec:action} 

In order to obtain a Milnor-type theorem, 
a key step is to give a set of representatives $\UU$. 
In this section, 
we give sets of representatives of some actions, 
which are given by some maximal parabolic subgroups. 
The result of this section is used essentially in the latter sections. 

In order to study our actions, we state two lemmas. 
The first one is a simple criteria for a subset $\UU \subset \GL_{p+q}(\R)$ to be a set of representatives. 
Let $H$ be a subgroup of $\GL_{p+q}(\R)$, 
and consider the action of $H$ on $\MM_{(p,q)} = \GL_{p+q} / \OO(p,q)$. 
We denote the double coset of $g \in \GL_{p+q}(\R)$ by 
\begin{align}
[[g]] := H g \, \OO(p,q) . 
\end{align}

\begin{Lem} 
\label{lem:representatives}
Consider an action of $H \subset \GL_{p+q}(\R)$ on $\MM_{(p,q)}$. 
Then, a subset $\UU \subset \GL_{p+q}(\R)$ is a set of representatives of this action 
if and only if for any $g \in \GL_{p+q}(\R)$, 
there exists $g_0 \in \UU$ such that $g_0 \in [[g]]$. 
\end{Lem} 

\begin{proof} 
This follows from a standard argument of double cosets. 
In fact, the proof for the Riemannian case can be found in \cite[Lemma~2.3]{HTT}, 
and the proof for this lemma is exactly the same. 
\end{proof} 

The second lemma is a general property of the natural action of $\OO(1,1)$. 

\begin{Lem} \label{lem:4-1}
Let $(x,y) \neq (0,0)$.
Then there exist $a>0$, $\lambda \in \{ 0,1,2 \}$, and $g \in \OO(1,1)$ such that 
$(x,y)\, g = (a , \lambda a)$ holds. 
\end{Lem}

\begin{proof}
We divide the proof into three cases. 
The first case is when $x^2 - y^2 > 0$. 
In this case, let us take 
\begin{align}
a := (x^2 - y^2)^{1/2} > 0, \quad 
        g := 
\frac{1}{a}
             \left( \begin{array}{cc} x & -y \\ -y & x \end{array} \right) 
\in \OO(1,1). 
\end{align}
Then one has $(x , y)\, g = (a , 0)$. 
Hence, by putting $\lambda := 0$, 
we complete the proof for this case. 

The second case is when $x^2 - y^2 = 0$. 
In this case, one can choose 
\begin{align}
g \in \left\{ 
\left( \begin{array}{cc} 
\pm 1 & 0 \\ 0 & \pm 1 
\end{array} \right) \right\} 
\subset \OO(1,1) 
\end{align}
such that $(x , y)\, g = (|x| , |y|)$. 
Hence, by putting $a := |x| = |y| > 0$ and $\lambda := 1$, 
we complete the proof for this case. 

The last case is when $x^2 - y^2 < 0$. 
In this case, let us take 
\begin{align}
a := ((y^2 - x^2)/3)^{1/2} > 0, \quad 
        g := 
\frac{1}{3a}
             \left( \begin{array}{cc} 2y-x & y-2x \\ y-2x & 2y-x \end{array} \right)
\in \OO(1,1). 
\end{align}
Then one has $(x , y)\, g = (a , 2a)$. 
Hence, by putting $\lambda := 2$, 
we complete the proof for this case. 
\end{proof} 

We now study actions on $\MM_{(p,q)}$. 
One action we consider in this section is given by the following group: 
\begin{align} 
Q_1 := \left\{ \left( 
\begin{array}{c|cccc} 
 \ast	&0	&\cdots	&0	\\\hline 
 \ast   &	&       &   	\\ 
 \vdots	&   	&\ast	&   	\\ 
 \ast	&   	&       &	\\ 
\end{array} 
\right) \in \GL_{p+q} (\R) \right\} , 
\end{align} 
where the size of the block decomposition is $(1 , p+q-1)$. 
This $Q_1$ is known to be a maximal parabolic subgroup of $\GL_{p+q}(\R)$. 
The following gives a set of representatives of the action of $Q_1$. 

\begin{Prop}
\label{prop:set-of-representatives}
Let $p,q \geq 1$, and consider the action of $Q_1$ on $\MM_{(p,q)}$. 
Then, the following $\UU$ is a set of representatives$:$ 
\begin{align}
\UU := 
\left\{ 
I_{p+q} + \lambda E_{1, p+q} \mid \lambda = 0,1,2 
\right\} . 
\end{align}
\end{Prop}

\begin{proof}
Take any $g = (g_{ij}) \in \GL_{p+q}(\R)$. 
According to Lemma~\ref{lem:representatives}, 
we have only to prove that 
there exists $g_0 \in \UU$ such that $g_0 \in [[g]]$. 
Recall that 
\begin{align*} 
[[g]] = Q_1 \, g \, \OO(p,q) . 
\end{align*} 
First of all, 
one knows that there exist $h_1 \in \OO(p)$ and $h_2 \in \OO(q)$ such that 
\begin{align}
(g_{11}, \ldots, g_{1 p}) \, h_1 = (x, 0, \ldots, 0) , \quad 
(g_{1(p+1)}, \ldots, g_{1 (p+q)}) \, h_2 = (0, \ldots, 0, y) . 
\end{align} 
Since $\OO(p) \times \OO(q)$ is naturally a subgroup of $\OO(p,q)$, we have 
\begin{align} 
[[g]] \ni g 
\left( \begin{array}{cc} 
h_1 & \\ & h_2 
\end{array} \right) 
= \left( 
\begin{array}{c|cccc}
 x	&0	&\cdots	&0      &y	\\\hline
 \ast	&\ast	&\cdots	&\cdots &\ast	\\
 \vdots	&\vdots	&\ddots	&	&\vdots	\\
 \vdots	&\vdots	&	&\ddots	&\vdots	\\
 \ast	&\ast	&\cdots	&\cdots &\ast	\\
\end{array} 
\right) =: g_1 . 
\end{align} 
One has $(x,y) \ne (0,0)$, since $\det (g_1) \neq 0$. 
Furthermore, $\OO(1,1)$ is naturally a subgroup of $\OO(p,q)$, since $p,q \geq 1$. 
Hence, Lemma~\ref{lem:4-1} yields that 
there exist $a>0$, $\lambda \in \{ 0,1,2 \}$, and $k \in \OO(p,q)$ such that 
\begin{align} 
[[g]] \ni g_1 k = \left( 
\begin{array}{c|cccc} 
 a      &0	&\cdots	&0	&a\lambda \\\hline 
 a_2	&\ast	&\cdots	&\cdots &\ast	\\
 \vdots	&\vdots	&\ddots	&	&\vdots	\\
 \vdots	&\vdots	&	&\ddots	&\vdots	\\
 a_{p+q}	&\ast	&\cdots	&\cdots &\ast	\\
\end{array} 
\right) =: g_2 , 
\end{align} 
where $a_2, \ldots, a_{p+q} \in \R$. 
It then follows from the definition of $Q_1$ that 
\begin{align}
[[g]] \ni 
\left(
\begin{array}{c|ccc}
 1/a	&0	&\cdots	&0	\\\hline
 -a_2   &a	&       &   	\\
 \vdots &   	&\ddots	&   	\\
 -a_{p+q}   &   	&       &a	\\
\end{array}
\right) 
g_2 = \left(
\begin{array}{c|ccc}
 1	&0	&\cdots	&0	\\\hline
 0   	&	&       &   	\\
 \vdots	&   	&\alpha	&   	\\
 0	&   	&       &	\\
\end{array}
\right) + \lambda E_{1,{p+q}} =: g_3 . 
\end{align}
Since $0 \neq \det (g_3) = \det(\alpha)$, we conclude that 
\begin{align} 
[[g]] \ni 
\left( 
\begin{array}{c|ccc}
 1	&0	&\cdots	&0	\\\hline
 0   	&	&       &   	\\
 \vdots	&   	&\alpha^{-1}	&   	\\
 0	&   	&       &	\\
\end{array} 
\right) g_3 = I_{p+q} + \lambda E_{1, p+q} =: g_0 . 
\end{align} 
One can see that $g_0 \in \UU$, which completes the proof. 
\end{proof}

Another action we study is given by the transpose of $Q_1$. 
Note that, if $H$ is a subgroup of $\GL_{p+q}(\R)$, then so is 
\begin{align} 
H' := \{ {}^t h \mid h \in H \} . 
\end{align} 
By the following proposition and Proposition~\ref{prop:set-of-representatives}, 
one can obtain a set of representatives of the action of $Q'_1$. 

\begin{Prop} 
\label{prop:dual_action} 
Let $H$ be a subgroup of $\GL_{p+q}(\R)$, 
and $\UU$ be a set of representatives of the action of $H$ on $\MM_{(p,q)}$. 
Then, the following $\UU^\ast$ is a set of representative of the action of $H'$ on $\MM_{(p,q)}$$:$ 
\begin{align} 
\UU^\ast := \{ {}^t u^{-1} \mid u \in \UU \} . 
\end{align} 
\end{Prop} 

\begin{proof} 
Take any $g \in \GL_{p+q}(\R)$. 
In this proof, 
we write the double cosets by $H g \, \OO(p,q)$ and $H' g \, \OO(p,q)$, 
in order to distinguish them. 
Then, according to Lemma~\ref{lem:representatives}, 
we have only to prove that 
there exists $g_0 \in \UU^\ast$ such that $g_0 \in H' g \, \OO(p,q)$. 
Since ${}^t g^{-1} \in \GL_n(\R)$ and $\UU$ is a set of representatives of the action of $H$, 
there exists $u \in \UU$ such that $u \in H ({}^t g^{-1}) \OO(p,q)$. 
That is, one can write 
\begin{align} 
u = h \cdot {}^t g^{-1} \cdot k \quad (h \in H , \ k \in \OO(p,q)) . 
\end{align} 
We put $g_0 := {}^t u^{-1} \in \UU^\ast$. 
One thus has 
\begin{align} 
g_0 = {}^t u^{-1} = {}^t h^{-1} \cdot g \cdot {}^t k^{-1} \in H' g \, \OO(p,q) , 
\end{align} 
since $h^{-1} \in H$ and ${}^t k^{-1} \in \OO(p,q)$. 
This completes the proof. 
\end{proof}

\section{On the Lie groups of real hyperbolic spaces} 
\label{section:RHn} 

In this section, 
we study $G_{\RH^n}$, the Lie group of real hyperbolic space, 
and prove Theorem~\ref{thm:MTT1} and Corollary~\ref{cor:cor1}. 

First of all, 
we show that $\R^{\times} \Aut(\LG_{\RH^n})$ coincides with $Q_1$, studied in the previous section. 
We use the canonical basis $\{ e_1 , \ldots , e_n \}$ of $\LG_{\RH^n}$, 
which satisfies 
\begin{equation} 
\label{eq:bracket_def} 
[e_1, e_i] = e_i\qquad (\forall i \in \{2, \ldots , n\}) . 
\end{equation} 

\begin{Prop} 
\label{prop:RAut-RHn} 
The matrix expression of $\R^{\times} \Aut(\LG_{\RH^n})$ 
with respect to the canonical basis $\{ e_1 , \ldots , e_n \}$ coincides with $Q_1$. 
\end{Prop} 

\begin{proof} 
We put $\LG := \LG_{\RH^n}$ for simplicity. 
It is sufficient to show 
\begin{align} 
\label{eq:aut-RHn} 
\Aut(\LG) = \left\{ \left( 
\begin{array}{c|cccc} 
 1	&0	&\cdots	&0	\\\hline
 \ast   &	&       &   	\\
 \vdots	&   	&\ast	&   	\\
 \ast	&   	&       &	\\
\end{array}
\right) \in \GL_n(\R) \right\} =: H_0.
\end{align}
In order to show this, we claim that 
\begin{align} \label{lem:4-1_eq1} 
[ge_1, ge_j] - g[e_1, e_j] = (g_{11} - 1) g e_j 
\quad (\forall g = (g_{ij}) \in Q_1 , \ \forall j \in \{ 2, \ldots, n \}) . 
\end{align} 
In fact, for such $g$ and $j$, one has 
$g e_j \in \mathrm{span} \{e_2, \ldots, e_n \}$, and hence 
\begin{align} 
\textstyle 
[ge_1, ge_j] = [\sum_{l=1}^n g_{l1} e_l , g e_j ] = [g_{11} e_1 , g e_j] = g_{11} g e_j . 
\end{align} 
This proves the above claim. 

We show $\Aut(\LG) \subset H_0$. 
Take any $g = (g_{ij}) \in \Aut(\LG)$. 
We need to show $g \in Q_1$ and $g_{11} = 1$. 
One has $g \in Q_1$, since $\Aut(\LG)$ preserves 
\begin{align} 
[\LG, \LG] = \mathrm{span} \{e_2, \ldots, e_n \} . 
\end{align} 
Furthermore, 
since $g \in \Aut(\LG)$, one has from (\ref{lem:4-1_eq1}) that 
\begin{align}
0 = (g_{11} - 1) g e_2 . 
\end{align}
Note that $g e_2 \neq 0$. 
We thus have $g_{11} = 1$, which shows $g \in H_0$. 

It remains to show $H_0 \subset \Aut(\LG)$. 
Take any $g = (g_{ij}) \in H_0$. 
Since $g$ preserves $\mathrm{span} \{e_2, \ldots, e_n \}$, one has 
\begin{align} 
g[e_j, e_k] = 0 = [ge_j, ge_k] \quad (\forall j, k \in \{ 2, \ldots, n \}) . 
\end{align}
Furthermore, since $g \in H_0 \subset Q_1$ and $g_{11} = 1$, 
one can see from (\ref{lem:4-1_eq1}) that 
\begin{align} 
[ge_1, ge_j] - g[e_1, e_j] = 0 \quad (\forall j \in \{ 2, \ldots, n \}) . 
\end{align}
This concludes $g \in \Aut(\LG)$, which completes the proof. 
\end{proof}

By Proposition~\ref{prop:set-of-representatives}, 
one has a set of representatives of the action of $Q_1$ on $\MM_{(p,q)}$, 
consisting of three points. 
We here prove Theorem~\ref{thm:MTT1} in terms of this set of representatives. 

\begin{proof}[Proof of Theorem~\ref{thm:MTT1}] 
Put $n := p+q$ with $p,q \geq 1$, 
and take an inner product $\inner{}{}$ of signature $(p,q)$ on $\mathfrak{g}_{\R \mathrm{H}^n}$. 
By Propositions~\ref{prop:set-of-representatives} and \ref{prop:RAut-RHn}, 
one knows that 
\begin{align} 
\mathfrak{U} := \{ I_n + \lambda E_{1,n} \mid \lambda = 0, 1, 2 \} 
\end{align} 
is a set of representatives of the action of $\R^\times \Aut(\LG_{\RH^n})$ 
with respect to $\{ e_1, \ldots , e_n \}$. 
Hence, it follows from 
Theorem~\ref{thm:pseudo-key}
that there exist $k > 0$, 
$\varphi \in \mathrm{Aut}(\mathfrak{g}_{\R \mathrm{H}^n})$, 
and $g_0 \in \mathfrak{U}$ such that $\{\varphi g_0 e_1, \ldots , \varphi g_0 e_n\}$ is pseudo-orthonormal 
with respect to $k \inner{}{}$. 
By the definition of $\mathfrak{U}$, 
there exists $\lambda \in \{0,1,2\}$ such that $g_0 = I_n + \lambda E_{1,n}$. 
We here put 
\begin{align} 
x_i := \varphi g_0 e_i \quad (i \in \{1, \ldots , n\}) . 
\end{align} 
Since $\{x_1, \ldots , x_n\}$ is 
pseudo-orthonormal 
with respect to $k \inner{}{}$, 
we have only to
show that the bracket relations among them are given by (i)--(iii) of Theorem~\ref{thm:MTT1}. 
We use 
\begin{equation}
\label{eq:lambda}
g_0 e_i = e_i \quad (\forall i \in \{1, \ldots , n-1\}), \qquad g_0 e_n = \lambda e_1 + e_n . 
\end{equation} 

First of all, we prove that the relation (i) holds. 
Take any $i \in \{2, \ldots , n-2\}$. 
Then it follows from
(\ref{eq:lambda}) 
that 
\begin{equation}
[g_0 e_1, g_0 e_i] = [e_1, e_i] = e_i = g_0 e_i . 
\end{equation}
Since $\varphi$ is an automorphism, one obtains (i) by
\begin{equation}
[x_1, x_i] = \varphi [g_0 e_1, g_0 e_i] = \varphi g_0 e_i = x_i . 
\end{equation}

We next
show that the relation (ii) holds. 
It follows from (\ref{eq:lambda}) that 
\begin{equation}
\label{eq:bra_1n}
[g_0 e_1, g_0 e_n] = [e_1, \lambda e_1 + e_n] 
= e_n 
= - \lambda e_1 + g_0 e_n = - \lambda g_0 e_1 + g_0 e_n . 
\end{equation}
By applying $\varphi$ to the both sides, one obtains (ii). 

We prove that the relation (iii) holds. 
Take any $i \in \{2, \ldots , n-1\}$. 
Then, (\ref{eq:lambda}) yields
that 
  \begin{equation}
    [g_0 e_i , g_0 e_n] = [e_i, \lambda e_1 + e_n] = -\lambda e_i = -\lambda g_0 e_i.
  \end{equation}
By applying $\varphi$ to the both sides, one obtains (iii). 

It remains to verify that the others bracket relations precisely vanish. 
Take any $i,j \in \{ 2, \ldots , n-1 \}$. 
Then it follows from (\ref{eq:lambda}) that  
  \begin{equation}
    [x_i, x_j] = \varphi [g_0 e_i, g_0 e_j] = \varphi  [e_i,  e_j] = 0 . 
  \end{equation}
This completes the proof of the proposition. 
\end{proof}

We next study the curvature properties of 
an arbitrary left-invariant metrics $\inner{}{}$ of signature $(p,q)$ on $G_{\RH^{p+q}}$, 
in terms of the basis $\{ x_1 , \ldots , x_{p+q} \}$ given in Theorem~\ref{thm:MTT1}. 
We put $n := p+q$, and calculate the curvatures under the normalization $k=1$. 
Recall that 
\begin{align} 
\varepsilon_i = \inner{x_i}{x_i} \in \{ \pm 1 \} . 
\end{align} 
First of all, 
we calculate the symmetric operator $U : \LG_{\RH^n} \times \LG_{\RH^n} \to \LG_{\RH^n}$ defined by, 
for any $X,Y,Z \in \LG_{\RH^n}$, 
\begin{align} 
2 \inner{U(X,Y)}{Z} = \inner{[Z,X]}{Y} + \inner{X}{[Z,Y]} . 
\end{align} 
Throughout the following calculations, let $i,j \in \{ 2, \ldots , n-1 \}$. 
Then, one can see from the bracket relations given in Theorem~\ref{thm:MTT1} that 
\begin{align} 
\begin{split} 
& 
U(x_1 , x_1) = \lambda \varepsilon_1 \varepsilon_n x_n , \qquad 
U(x_1 , x_i) = - (1/2) x_i , \\ 
& 
U(x_1 , x_n) = - (\lambda/2) x_1 - (1/2) x_n , \\ 
& 
U(x_i, x_j) = \delta_{ij} \varepsilon_i (\varepsilon_1 x_1 + \lambda \varepsilon_n x_n) , \\ 
& 
U(x_i, x_n) = - (1/2) \lambda x_i , \qquad 
U(x_n , x_n) = \varepsilon_1 \varepsilon_n x_1 . 
\end{split} 
\end{align} 
Recall that the Levi-Civita connection $\nabla$ of $(\LG_{\RH^n}, \inner{}{})$ is defined by 
\begin{align} 
\nabla_X Y := (1/2) [X,Y] + U(X,Y) \quad (X,Y \in \LG_{\RH^n}) . 
\end{align} 
Thus, one can directly calculate that 
\begin{align}
\label{nabla} 
\begin{split}
& 
\nabla_{x_{1}} x_{1} = \lambda \varepsilon _1 \varepsilon_n x_n , \quad 
\nabla_{x_{1}} x_{i} = 0 , \quad 
\nabla_{x_{1}} x_{n} = - \lambda x_1 , \\ 
& 
\nabla_{x_i} x_1 = - x_i , \quad 
\nabla_{x_i} x_{j} = \delta_{i j} \varepsilon _i (\varepsilon _1 x_1 + \lambda \varepsilon _n x_n) , \quad 
\nabla_{x_i} x_{n} = - \lambda x_i , \\ 
& 
\nabla_{x_{n}} x_{1} = - x_n , \quad 
\nabla_{x_{n}} x_{i} = 0 , \quad 
\nabla_{x_{n}} x_{n} = \varepsilon _1 \varepsilon_n x_1 . 
\end{split} 
\end{align} 
We next calculate the curvature tensor $R$, taken with the sign convention 
\begin{equation}
R(X,Y) = [\nabla_{X} , \nabla_{Y}] - \nabla _{[X,Y]} . \label{curv} 
\end{equation} 
In order to express $R$, 
we use the linear map $X \wedge Y : \LG_{\RH^n} \to \LG_{\RH^n}$, 
where $X, Y \in \LG_{\RH^n}$, 
defined by 
\begin{align} 
(X \wedge Y) Z = \inner{Y}{Z} X - \inner{X}{Z} Y . 
\end{align} 

\begin{Prop} 
We keep the above notations. 
Then, for every $X, Y \in \LG_{\RH^n}$, the curvature tensor $R$ satisfies 
\begin{align} 
R(X,Y) = - (\lambda^2 \varepsilon_n + \varepsilon_1) X \wedge Y . \label{curv2} 
\end{align} 
\end{Prop} 

\begin{proof} 
First of all, we calculate $R(x_1 , x_i)$. 
By using \eqref{nabla}, one can directly see that 
\begin{align}
\begin{split} 
R(x_1, x_i) x_1 & =  (\lambda^2 \varepsilon_n +\varepsilon_1 ) \varepsilon_1 x_i, \\ 
R(x_1, x_i) x_j & =  -\delta _{i j} (\lambda^2 \varepsilon_n +\varepsilon_1 ) \varepsilon_i x_1, \\  
R(x_1, x_i) x_n & = 0 . 
\end{split} 
\end{align} 
This yields that 
\begin{align} 
R(x_1, x_i) =  - (\lambda^2 \varepsilon_n + \varepsilon_1 ) \varepsilon_1 x_1 \wedge x_i . 
\end{align} 
Similarly, one can directly calculate that 
\begin{align}
\begin{split} 
R(x_1, x_n) x_1 & =  (\lambda^2 \varepsilon_n + \varepsilon_1 ) \varepsilon_1 x_n , \\ 
R(x_1, x_n) x_i & = 0 , \\ 
R(x_1, x_n) x_n & =  - (\lambda^2 \varepsilon_n + \varepsilon_1 ) \varepsilon_n x_1 . 
\end{split} 
\end{align} 
This shows that $R(x_1 , x_n)$ agrees with \eqref{curv2}. 
One can also show the case for $R(x_i , x_j)$ by 
\begin{align} 
\begin{split} 
R(x_i, x_j) x_1 & = 0 , \\ 
R(x_i, x_j) x_k & = (\lambda^2 \varepsilon_n + \varepsilon_1 ) 
( \delta _{i k}  \varepsilon _i x_j -  \delta _{j k}  \varepsilon _j x_i) , \\  
R(x_i, x_j) x_n & = 0 . 
\end{split} 
\end{align} 
Finally, the case of $R(x_i, x_n)$ can be checked by 
\begin{align} 
\begin{split} 
R(x_i, x_n) x_1 & = 0 , \\ 
R(x_i, x_n) x_j & = \delta _{i j} (\lambda^2 \varepsilon_n +\varepsilon_1 ) \varepsilon_i x_n , \\ 
R(x_i, x_n) x_n & = - (\lambda^2 \varepsilon_n + \varepsilon_1 ) \varepsilon_n x_i . 
\end{split} 
\end{align} 
Since $R$ is skew-symmetric and bilinear, this completes the proof. 
\end{proof}

This proposition shows that the curvature tensor of $\LG_{\RH^{n}}$ has a simple form. 
We are now in position to complete the proof of Corollary~\ref{cor:cor1}. 

\begin{proof}[Proof of Corollary~\ref{cor:cor1}] 
Let $P$ be a nondegenerate tangent plane in $\LG_{\RH^{n}}$ with a basis $\{ X, Y \}$. 
Recall that the sectional curvature $K$ of $P$ is defined by 
\begin{align} 
K = \dfrac{\inner{R(X,Y)Y}{X}}{\inner{X}{X} \inner{Y}{Y}-\inner{X}{Y}^2} . 
\end{align} 
Then it follows from \eqref{curv2} that 
$(\LG_{\RH^n}, \inner{}{} )$ has constant sectional curvature 
$-(\lambda^2 \varepsilon_n + \varepsilon_1)$. 
This proves the first assertion. 

Since $p,q \geq 1$ by assumption, 
we have $\varepsilon_1=1$ and $\varepsilon_n = -1$. 
Hence, if $\lambda = 0$, $1$, or $2$, 
then we obtain the constant sectional curvature $-1$, $0$, or $3$, respectively. 
Note that they are the sectional curvatures under the normalization $k=1$. 
If $k$ varies over all positive real numbers, 
then the constant sectional curvature can take any real number. 
This proves the second assertion. 
\end{proof}

\begin{remark} 
The above arguments show that there are exactly three left-invariant pseudo-Riemannian metrics 
of signature $(p,q)$ on $G_{\RH^{p+q}}$ up to isometry and scaling, if $p,q \geq 1$. 
In fact, Theorem~\ref{thm:MTT1} yields that there are at most three, 
and Corollary~\ref{cor:cor1} shows that they cannot be isometric up to scaling, 
since the sign of the constant sectional curvatures are different. 
\end{remark}

\section{On the three-dimensional Heisenberg group} 
\label{section:H3} 

In this section, 
we apply our procedure to the three-dimensional Heisenberg group $H_3$ with Lie algebra $\LH_3$. 
Recall that Rahmani (\cite{R92}) proved that any left-invariant Lorentzian metric on $H_3$ 
can be classified into three types. 
Our argument gives an alternative proof of this fact. 

Throughout this section, 
we fix the canonical basis $\{ e_1 , e_2 , e_3 \}$ of $\LG = \LH_3$, 
whose bracket relations are given by 
\begin{align} 
[e_2 , e_3] = e_1 . 
\end{align} 
First of all, 
we describe $\R^{\times} \Aut(\LH_3)$ in terms of this basis. 
Note that the transpose of $Q_1$ in $\GL_3(\R)$ is given by 
\begin{align} 
Q'_1 = \left\{ \left(
\begin{array}{ccc}
\ast & \ast & \ast \\ 
0 & \ast & \ast \\ 
0 & \ast & \ast 
\end{array}
\right) \in \GL_3 (\R) \right\} . 
\end{align} 

\begin{Lem}
\label{lem:5-1} 
The matrix expression of $\R^{\times} \Aut(\LH_3)$ 
with respect to the canonical basis $\{ e_1 , e_2 , e_3 \}$ coincides with $Q'_1$. 
\end{Lem}

\begin{proof} 
One can directly show that the matrix expression of $\Aut(\LH_3)$ 
with respect to $\{ e_1 , e_2  , e_3 \}$ is 
\begin{align} 
\Aut(\LH_3) = \left\{ \left(
\begin{array}{ccc}
ad-bc & \ast & \ast \\ 
0 & a & b \\ 
0 & c & d 
\end{array}
\right) \mid ad-bc \neq 0 \right\} . 
\end{align} 
The lemma is an easy consequence of it. 
\end{proof}

Recall that $\UU = \{ I_3 + \lambda E_{1,3} \mid \lambda = 0 , 1 , 2 \}$ 
is a set of representatives of the action of $Q_1 \subset \GL_3(\R)$. 
Thus, by Proposition~\ref{prop:dual_action}, 
\begin{align} 
\label{eq:U-ast}
\UU^\ast := \{ {}^t u^{-1} \mid u \in \UU \} = \{ I_3 - \lambda E_{3,1} \mid \lambda = 0 , 1 , 2 \} 
\end{align} 
is a set of representatives of the action of $Q'_1 = \R^\times \Aut(\LH_3)$. 
In terms of this set of representatives, 
we can prove Theorem~\ref{thm:MTT2}.

\begin{proof}[Proof of Theorem~\ref{thm:MTT2}] 
Take an arbitrary inner product $\inner{}{}$ of signature $(2,1)$ on $\LH_3$. 
Since $\UU^\ast$ given in (\ref{eq:U-ast}) is a set of representatives, 
Theorem~\ref{thm:pseudo-key} yields that 
there exist $k > 0$, $\varphi \in \Aut(\LH_3)$, and $g_0 \in \UU^\ast$ such that 
$\{ \varphi g_0 e_1 , \varphi g_0 e_2 , \varphi g_0 e_3 \}$ 
is pseudo-orthonormal with respect to $k \inner{}{}$. 
Let us put 
\begin{align} 
x_i := \varphi g_0 e_i \quad (i \in \{ 1, 2, 3 \}) , 
\end{align} 
and study their bracket relations. 
One knows that there exists $\lambda \in \{ 0, 1, 2 \}$ such that $g_0 = I_3 - \lambda E_{3,1}$. 
We thus have 
\begin{align} 
g_0 e_1 = e_1 - \lambda e_3 , \quad 
g_0 e_2 = e_2 , \quad 
g_0 e_3 = e_3 . 
\end{align} 
This yields that 
\begin{align} 
[x_1 , x_2] = \varphi [ e_1 - \lambda e_3 , e_2 ] 
= \varphi \lambda e_1 
= \varphi \lambda (g_0 e_1 + \lambda g_0 e_3) 
= \lambda (x_1 + \lambda x_3) . 
\end{align} 
Other bracket products can be calculated similarly. 
\end{proof} 

We here recall the classification of left-invariant Lorentzian metrics on $H_3$ by Rahmani (\cite{R92}). 
In fact, Propositions~2.4 and 2.5 in \cite{R92} can be summarized as follows. 

\begin{Thm}[\cite{R92}] 
\label{thm:R92} 
Let $\inner{}{}$ be an arbitrary inner product on $\LH_3$ of signature $(2,1)$. 
Then, there exists a pseudo-orthonormal basis $\{ f_1 , f_2 , f_3 \}$ with respect to $\inner{}{}$ 
such that one of the following bracket relations holds$:$ 
\begin{enumerate} 
\item
$[f_2 , f_3] = \alpha f_1$, 
$[f_3 , f_1] = 0$, 
$[f_2 , f_1] = 0$, where $\alpha>0$, 
\item
$[f_2 , f_3] = 0$, 
$[f_3 , f_1] = 0$, 
$[f_2 , f_1] = \gamma f_3$, where $\gamma>0$, 
\item
$[f_2 , f_3] = 0$, 
$[f_3 , f_1] = f_2 - f_3$, 
$[f_2 , f_1] = f_2 - f_3$. 
\end{enumerate} 
\end{Thm} 

He denotes by $g_1$, $g_2$, and $g_3$ 
the left-invariant Lorentzian metrics corresponding to (1), (2), and (3), respectively. 
Finally in this section, 
we compare the above result to our Theorem~\ref{thm:MTT2}, 
and review some known curvature properties. 

\begin{remark} 
Let $\inner{}{}$ be an arbitrary inner product on $\LH_3$ of signature $(2,1)$. 
Then, according to Theorem~\ref{thm:MTT2}, 
we have three cases, namely $\lambda = 0,1,2$. 
\begin{enumerate} 
\item 
The case of $\lambda = 0$. 
Then, under a certain scaling ($k=1$), 
there exists a pseudo-orthonormal basis $\{ x_1 , x_2 , x_3 \}$ with respect to $\inner{}{}$ such that 
\begin{align*} 
[x_1 , x_2] = 0 , \quad [x_2 , x_3] = x_1 , \quad [x_3 , x_1] = 0 . 
\end{align*} 
Therefore, 
this pseudo-orthonormal basis satisfies the same bracket relations 
as in Theorem~\ref{thm:R92} (1) with $\alpha = 1$, 
and hence the metric coincides with $g_1$ in the above notation. 
It has been known that $g_1$ is not Einstein, but algebraic Ricci soliton 
(\cite{Onda2}, see also \cite{Onda}). 

\item 
The case of $\lambda = 1$. 
Then, up to scaling, 
there exists a pseudo-orthonormal basis $\{ x_1 , x_2 , x_3 \}$ with respect to $\inner{}{}$ such that 
\begin{align*} 
[x_1 , x_2] = x_1 + x_3 , \quad [x_2 , x_3] = x_1 + x_3 , \quad [x_3 , x_1] = 0 . 
\end{align*} 
Put $f_1 := x_2$, $f_2 := - x_1$, and $f_3 := x_3$. 
Then one can directly check that 
$\{ f_1 , f_2 , f_3 \}$ is pseudo-orthonormal and 
satisfies the same bracket relations as in Theorem~\ref{thm:R92} (3). 
Then the metric coincides with $g_3$, 
which is known to be flat (\cite{Nomizu}, see also \cite{RR06}). 

\item 
The case of $\lambda = 2$. 
Then, up to scaling, 
there exists a pseudo-orthonormal basis $\{ x_1 , x_2 , x_3 \}$ with respect to $\inner{}{}$ such that 
\begin{align*} 
[x_1 , x_2] = 2 (x_1 + 2 x_3) , \quad [x_2 , x_3] = x_1 + 2 x_3 , \quad [x_3 , x_1] = 0 . 
\end{align*} 
We define a pseudo-orthonormal basis $\{ f_1 , f_2 , f_3 \}$ by 
\begin{align*} 
f_1 := x_2 , \quad 
f_2 := 3^{-1/2} (2 x_1 + x_3) , \quad 
f_3 := 3^{-1/2} (x_1 + 2 x_3) . 
\end{align*} 
One can directly show that 
it satisfies the same bracket relations as in Theorem~\ref{thm:R92} (2) with $\gamma = 3$. 
Then the metric corresponds to $g_2$, 
which is not Einstein, but algebraic Ricci soliton 
(\cite{Onda2}, see also \cite{Onda}). 
\end{enumerate} 
\end{remark} 

For the notation of algebraic Ricci soliton metrics in the pseudo-Riemannian case, 
we refer to \cite{Onda2} and \cite{BO11a}. 
Note that algebraic Ricci soliton metrics give rise to left-invariant Ricci soliton metrics.


\begin{thebibliography}{99} 

\bibitem{BO11a} 
Batat, W., Onda, K.: 
\textit{Algebraic Ricci solitons of three-dimensional Lorentzian Lie groups}. 
Preprint, arXiv:1112.2455v2. 

\bibitem{Chebarikov} 
Chebarikov, M.~S.: 
\textit{On the Ricci curvature of three-dimensional metric Lie algebras (Russian)}. 
Vladikavkaz.\ Mat.\ Zh. \textbf{16} (2014), 57--67 

\bibitem{CP} 
Cordero, L.~A.,  Parker, P.~E.: 
\textit{Left-invariant Lorentzian metrics on 3-dimensional Lie groups}. 
Rend.\ Mat., Serie VII \textbf{17} (1997), 129--155 

\bibitem{HL} 
Ha, K.~Y., Lee, J.~B.: 
\textit{Left invariant metrics and curvatures on simply connected three-dimensional Lie groups}. 
Math.\ Nachr. \textbf{282} (2009), 868--898 


\bibitem{HT} 
Hashinaga, T., Tamaru, H.: 
\textit{Three-dimensional solvsolitons and the minimality of the corresponding submanifolds}. 
Preprint, arXiv:1501.05513. 


\bibitem{HTT} 
Hashinaga, T., Tamaru, H., Terada, K.: 
\textit{Milnor-type theorems for left-invariant Riemannian metrics on Lie groups}. 
J.\ Math.\ Soc.\ Japan, to appear, 
arXiv:1501.02485 

\bibitem{Helgason} 
Helgason, S.: 
\textit{Differential geometry and symmetric spaces}. 
Pure and Applied Mathematics \textbf{XII}, 
Academic Press, New York-London, 1962 

\bibitem{KTT} 
Kodama, H., Takahara, A., Tamaru, H.: 
\textit{The space of left-invariant metrics on a Lie group up to isometry and scaling}. 
Manuscripta Math.\ \textbf{135} (2011), 229--243 

\bibitem{KN1}
Kremlev, A.~G., Nikonorov, Yu.~G.: 
\textit{The Signature of the Ricci Curvature of Left-Invariant Riemannian
Metrics on Four-Dimensional Lie Groups. The Unimodular Case}. 
Siberian Adv.\ Math.\ \textbf{19} (2009), 245--267 

\bibitem{KN2}
Kremlev, A.~G., Nikonorov, Yu.~G.: 
\textit{The Signature of the Ricci Curvature of Left-Invariant Riemannian
Metrics on Four-Dimensional Lie Groups. The Nonunimodular Case}. 
Siberian Adv.\ Math.\ \textbf{20} (2010), 1--57 

\bibitem{L}
Lauret, J.: 
\textit{Degenerations of Lie algebras and geometry of Lie groups}. 
Differential Geom.\ Appl. \textbf{18} (2003), no.\ 2, 177--194 

\bibitem{Mi}
Milnor, J.:
\textit{Curvatures of left invariant metrics on Lie groups}. 
Advances in Math.\ \textbf{21} (1976), no.\ 3, 293--329 

\bibitem{Nomizu}
Nomizu, K.: 
\textit{Left-invariant Lorentz metrics on Lie groups}. 
Osaka J.\ Math.\ \textbf{16} (1979), 143--150 

\bibitem{O83} 
O'Neill, B.: 
\textit{Semi-Riemannian geometry with applications to relativity}. 
Academic Press.\ NewYork (1983) 

\bibitem{Onda} 
Onda, K.: 
\textit{Lorentz Ricci solitons on $3$-dimensional Lie groups}.  
Geom.\ Dedicata \textbf{147} (2010), 313--322

\bibitem{Onda2} 
Onda, K.: 
\textit{Examples of algebraic Ricci solitons in the pseudo-Riemannian case}. 
Acta Math.\ Hungar. 
\textbf{144} (2014), 247--265 

\bibitem{R92}
Rahmani, S.: 
\textit{Metriques de Lorentz sur les groupes de Lie unimodulaires, de dimension trois}. 
J.\ Geom.\ Phys.\  \textbf{9} (1992), no.~3, 295--302 

\bibitem{RR06}
Rahmani, N., Rahmani, S.: 
\textit{Lorentzian geometry of the Heisenberg group}. 
Geom.\ Dedicata \textbf{118} (2006), 133--140 

\bibitem{TT} 
Taketomi, Y., Tamaru, H.: 
\textit{On the nonexistence of left-invariant Ricci solitons --- a conjecture and examples}. 
Preprint. 

\end{thebibliography}
\end{document}